\documentclass[12pt,leqno,draft]{article}
\usepackage{amsfonts}
\usepackage{amsfonts}
\usepackage{amsmath, amsthm, amsfonts, amssymb, color}
\usepackage{mathrsfs}
\usepackage{color}
\usepackage{stmaryrd}
\makeatletter
\newcommand{\Spvek}[2][r]{%
  \gdef\@VORNE{1}
  \left(\hskip-\arraycolsep%
    \begin{array}{#1}\vekSp@lten{#2}\end{array}%
  \hskip-\arraycolsep\right)}

\def\vekSp@lten#1{\xvekSp@lten#1;vekL@stLine;}
\def\vekL@stLine{vekL@stLine}
\def\xvekSp@lten#1;{\def\temp{#1}%
  \ifx\temp\vekL@stLine
  \else
    \ifnum\@VORNE=1\gdef\@VORNE{0}
    \else\@arraycr\fi%
    #1%
    \expandafter\xvekSp@lten
  \fi}
\makeatother

\newtheorem{thm}{Theorem}[section]
\newtheorem{cor}[thm]{Corollary}
\newtheorem{lem}[thm]{Lemma}

\newtheorem{exa}[thm]{Example}
\newtheorem{rem}[thm]{Remark}
\theoremstyle{definition}
\newtheorem{defn}{Definition}[section]
\newcommand{\scr}[1]{\mathscr #1}
\definecolor{wco}{rgb}{0.5,0.2,0.3}

\numberwithin{equation}{section} \theoremstyle{remark}

\newcommand{\ua}{\uparrow}

\title{{\bf  Path Independence of Additive Functionals for SDEs  under $G$-framework}
}
\author{
{\bf Panpan Ren$^{b)}$,    Fen-Fen Yang$^{a)}\footnote{Corresponding author}$  }\\
\footnotesize{$^{b)}$Department of Mathematics, Swansea   University, Singleton Park, SA2 8PP, United Kingdom}\\
\footnotesize{$^{a)}$Center for Applied Mathematics, Tianjin University, Tianjin 300072, China}\\
\footnotesize{$^{b)}$673788@swansea.ac.uk}, $^{a)}$yangfenfen@tju.edu.cn}
\begin{document}
\allowdisplaybreaks
\def\R{\mathbb R}  \def\ff{\frac} \def\ss{\sqrt} \def\B{\mathbf
B}
\def\N{\mathbb N} \def\kk{\kappa} \def\m{{\bf m}}
\def\ee{\varepsilon}\def\ddd{D^*}
\def\dd{\delta} \def\DD{\Delta} \def\vv{\varepsilon} \def\rr{\rho}
\def\<{\langle} \def\>{\rangle} \def\GG{\Gamma} \def\gg{\gamma}
  \def\nn{\nabla} \def\pp{\partial} \def\E{\mathbb E}
\def\d{\text{\rm{d}}} \def\bb{\beta} \def\aa{\alpha} \def\D{\scr D}
  \def\si{\sigma} \def\ess{\text{\rm{ess}}}
\def\beg{\begin} \def\beq{\begin{equation}}  \def\F{\scr F}
\def\Ric{\text{\rm{Ric}}} \def\Hess{\text{\rm{Hess}}}
\def\e{\text{\rm{e}}} \def\ua{\underline a} \def\OO{\Omega}  \def\oo{\omega}
 \def\tt{\tilde} \def\Ric{\text{\rm{Ric}}}
\def\cut{\text{\rm{cut}}} \def\P{\mathbb P} \def\ifn{I_n(f^{\bigotimes n})}
\def\C{\scr C}   \def\G{\scr G}   \def\aaa{\mathbf{r}}     \def\r{r}
\def\gap{\text{\rm{gap}}} \def\prr{\pi_{{\bf m},\varrho}}  \def\r{\mathbf r}
\def\Z{\mathbb Z} \def\vrr{\varrho} \def\ll{\lambda}
\def\L{\scr L}\def\Tt{\tt} \def\TT{\tt}\def\II{\mathbb I}
\def\i{{\rm in}}\def\Sect{{\rm Sect}}  \def\H{\mathbb H}
\def\M{\scr M}\def\Q{\mathbb Q} \def\texto{\text{o}} \def\LL{\Lambda}
\def\Rank{{\rm Rank}} \def\B{\scr B} \def\i{{\rm i}} \def\HR{\hat{\R}^d}
\def\to{\rightarrow}\def\l{\ell}\def\iint{\int}
\def\EE{\scr E}\def\no{\nonumber}
\def\A{\scr A}\def\V{\mathbb V}\def\osc{{\rm osc}}
\def\BB{\scr B}\def\Ent{{\rm Ent}}\def\3{\triangle}\def\H{\scr H}
\def\U{\scr U}\def\8{\infty}\def\1{\lesssim}\def\HH{\mathrm{H}}
 \def\T{\scr T}
\maketitle

\begin{abstract}
The path independence of additive functionals for SDEs driven by the $G$-Brownian motion is characterized by nonlinear PDEs. The main result generalizes the existing ones for   SDEs driven by the standard Brownian motion.

\end{abstract} \noindent
 \vskip 2cm

\section{Introduction}

Stochastic differential equations (SDEs) under the linear
probability space have been widely used in
 modeling financial markets and economic phenomena \cite{B08,BS}. However, in many practical situations, most of the financial activities
 take place with uncertainty \cite{BR},  for which a fundamental theory of SDEs driven by the $G$-Brownian motion ($G$-SDEs) has been developed in
   \cite{peng2, peng1, peng4}. Since then $G$-SDEs have received much attention,  see for instance  \cite{HJP} on the Feyman-Kac formula,  \cite{HJ,HJY} on the stochastic control,  \cite{FWZ,FOZ} on the ergodicity,
 \cite{RYS,WG} on the
 stochastic stability, and \cite{GRR} on the  $G$-SPDEs.

In the equilibrium financial market, there exists a risk neutral
measure which admits
 a path independent density with respect to the real world
 probability
 \cite{HC}. To construct such risk neutral measures,  the path independence of additive functionals for
 SDEs has been investigated extensively; see \cite{Wang} for the   pioneer
 work.
 Subsequently, \cite{Wang}
 has been extended in   \cite{Wang3,QW,Wang2} for finite dimensional SDEs, and in \cite{QW2, Wang1} for infinite dimensional
 SPDEs, where \cite{Wang2} allows the SDEs involved to be degenerate. Recently, \cite{RW}
 investigated the path independence of additive functionals for a class of distribution dependent SDEs.
Nevertheless,  all of these papers only focus on linear probability
spaces.  To fill this gap,  in this   paper, we intend to characterize the  path
independence of additive functionals  for $G$-SDEs.
To this end, below we recall some basic facts on the $G$-Brownian motion.

For a positive integer $d$, let $(\R^d, \<\cdot,\cdot\>,|\cdot|)$ be the $d$-dimensional Euclidean space, $\R^d\otimes\R^d$  the family
of all $d\times d$-matrices, $\mathbb{S}^d$
the collection of all symmetric $d\times d$-matrices, ${\bf0}_d\in\R^d$ the zero vector, ${\bf0}_{d\times d}\in\R^d\otimes \R^d$ the zero matrix,
 and ${\bf{I}}_{d\times d}\in\R^d\otimes \R^d$ the identity  matrix.
 For a matrix $A$, let $A^*$ be its transpose and $\|A\|_{\rm HS}=(\mbox{trace}(AA^*))^{1/2}$ be its  Hilbert-Schmidt norm (or Frobenius norm). For a
 number $a\in\R$, $a^+$ and $a^-$ stipulate its positive part and  negative part, respectively.
  For $\sigma_1, \sigma_2 \in \mathbb{S}^d,$  the notation $\sigma_1\leq \sigma_2$ (res. $\sigma_1 < \sigma_2$) means
   that $\sigma_2- \sigma_1$ is non-negative (res. positive) definite, and we let
  $$ [{\sigma_1},  {\sigma_2}] :=\{\gamma|\gamma\in \mathbb{S}^d, {\sigma_1}\leq\gamma \leq  {\sigma_2}\}.$$
    Let $C^{1,2}(\R_+\times\R^d;\R)$ be the collection of
   all continuous functions $V:\R_+\times\R^d\rightarrow\R $ which are once differentiable w.r.t. the first argument, twice differentiable w.r.t.
    the second argument, and all these derivatives are joint continuous. Write $\nn$ and $\nn^2$ by the gradient operator and Hessian operator, respectively.

For any fixed $T>0$, $$\OO_T=\{\omega|[0,T]\ni
t\mapsto\omega_t\in\R^d \mbox{ is continuous with
}\omega(0)={\bf0}_d \}$$ endowed with the uniform topology.  Let
$B_t(\omega)=\omega_t, \omega\in\OO_T, $ be the canonical process.
Set
  $$L_{ip}(\Omega_T):=\{\varphi(B_{t_1}, \cdots, B_{t_n}),
  n\in\mathbb{N},
   t_1,\cdot \cdot \cdot, t_n\in [0,T],\varphi\in C_{b,lip}(\R^{d}\otimes\R^n)\},$$
where  $C_{b,lip}(\mathbb{R}^{d}\otimes\R^n)$ denotes  the set of
bounded Lipschitz functions  $f:\mathbb{R}^{d}\otimes \R^n\to\R$.
 Let $G: \mathbb{S}^d \to \R $ be a monotonic, sublinear   and homogeneous function; see e.g.
\cite[p16]{peng4}. Throughtout the paper,  we  always assume that $G:
\mathbb{S}^d \to \R $ is non-degenerate, i.e., there exists some
$\delta>0$ such that
  \begin{equation} \label{Gnon}
     G(A)-G(B)\geq \frac{\delta}{2} \mbox{trace}[A-B],\ A\geq B, A,B \in \mathbb{S}^d.
  \end{equation}
 For any $\xi \in L_{ip}(\Omega_T)$,  i.e.,
 $$\xi(\omega)=\varphi(\omega(t_1),\cdot \cdot \cdot,\omega(t_{n})), \ 0=t_0< t_1<\cdot \cdot \cdot<t_n=T,$$
  the conditional $G$-expectation is  defined by
  $$
 \bar{\mathbb{E}}_t[\xi]:=u_k(t,\omega(t);\omega(t_1),\cdot \cdot \cdot,\omega(t_{k-1}) ),  \  \xi \in L_{ip}(\Omega_T), \ t \in [t_{k-1}, t_k), \ k=1, \cdot\cdot\cdot, n,
 $$
 where $(t,x)\mapsto u_k(t,x; x_1, \cdot\cdot\cdot,x_{k-1})$, $k=1,\cdot \cdot\cdot,n$, solves the following $G$-heat equation
\begin{equation}\label{Gheat}
 \begin{cases}
 \partial_tu_k+G(\partial_x^2u_k)=0, \ (t,x) \in [t_{k-1}, t_k)\times \R^d, \ k=1,\cdot \cdot\cdot,n,\\
 u_k(t_k,x;x_1,\cdot\cdot\cdot,x_{k-1})= u_{k+1}(t_k,x;x_1,\cdot\cdot\cdot,x_{k-1},x_k), \ k=1, \cdot\cdot\cdot, n-1,\\
  u_n(t_n,x;x_1,\cdot\cdot\cdot,x_{n-1})=\varphi(x_1,\cdot\cdot\cdot,x_{n-1},x), \ k=n.
 \end{cases}
 \end{equation}
 Since $G$ is non-degenerate,  the solution of \eqref{Gheat} satisfies $u_k\in C^{1,2}(\R_+\times\R^d;\R);$ see  \cite[Appendix $C$, Theorem 4.5, p127]{peng4}.
The corresponding $G$-expectation of $\xi$ is defined by $\bar{\mathbb{E}}[\xi]=\bar{\mathbb{E}}_0[\xi]$.
Then the canonical process $B_t(\omega):=\omega_t$ is called  a  $G$-Brownian motion in  $(\Omega_T,L_{G}^p(\Omega_T),\bar{\mathbb{E}})$,
 where
 $L_G^{p}(\Omega_T$)  is the completion of $L_{ip}(\Omega_T)$  under the norm $(\bar{\mathbb{E}}[|\cdot|^p])^{\frac{1}{p}}$, $p\geq1. $ By  definition, we have
  $G(A)=\frac{1}{2}\bar{\mathbb{E}}\langle
 AB_1,B_1\rangle$, $A\in\mathbb{S}^d$.   The function $G$ is called the generating function corresponding of the $d$-dimensional
$G$-Brownian motion $(B_t)_{t\ge0}$. According to \cite{peng4}, there exists a bounded, convex, and closed subset $\Theta\subset \mathbb{S}^m$ such that
\begin{equation}\label{GA}
 G(A)=\frac{1}{2}\sup _{Q\in \Theta}\mbox{trace}[AQ], \ A\in\mathbb{S}^d.
\end{equation}
In particular, for 1-dimensional $G$-Brownian motion $(B_t)_{t\ge0}$, one has
$G(a)=(\overline{\sigma}^2a^+-\underline{\sigma}^2a^-)/2, a\in \mathbb{R},  $ where $\overline{\sigma}^2: = \bar{\mathbb{E}} [B^2_1]\geq -\bar{\mathbb{E}} [-B^2_1 ]= :\underline{\sigma}^2>0$.

 Let
\begin{equation*}\label{equa11} M_{G}^{p,0}([0,T])=
\Big\{\eta_t:=\sum_{j=0}^{N-1} \xi_{j} I_{[t_j, t_{j+1})}(t);
~\xi_{j}\in L_{G}^p(\Omega_{t_{j}}),
 N\in\mathbb{N},\ 0=t_0<t_1<\cdots <t_N=T \Big\}.
\end{equation*}
Let $M_G^p([0,T])$ and $H_G^p([0,T])$ be the completion of $M_G^{p,0}([0,T])$ under the norm
$$\|\eta\|_{M_G^p([0,T])}:=\left(\bar{\mathbb{E}}\int_{0}^{T}|\eta_{t}|^p\d t\right)^{\frac{1}{p}},\  \|\eta\|_{H_G^p([0,T])}:=\left\{\bar{\mathbb{E}}\left(\int_{0}^{T}|\eta_{t}|^2\d t\right)^\frac{p}{2}\right\}^{\frac{1}{p}},$$
respectively. We need to point out that if  $p=2$, then $M_G^p([0,T])=H_G^p([0,T])$. Denote by $M_G^p([0,T];\R^d)$   all $d$-dimensional stochastic processes $\eta_t=(\eta_t^1, \cdot \cdot \cdot, \eta_t^d),$ $t\geq0$
with $\eta_t^i\in M_G^p([0,T]).$  Let $  H_G^1([0,T]; \R^d) $  be all  $d$-dimensional stochastic processes $\zeta_t=(\zeta_t^{1}, \cdot\cdot\cdot,\zeta_t^{d}), t\geq0$ with $\zeta^{i}\in H_G^1([0,T]). $

Furthermore,   we also need the  Choquet capacity associated with the $G$-expectation. Let $\mathcal{M}$ be the collection of all probability measures on  $(\Omega_T, \mathcal{B}(\Omega_T))$.
According to \cite{15},    there exists a weakly compact subset $\mathcal{P}\subset \mathcal{M}$   such that
$$\bar{\mathbb{E}}[X]=\sup_{P\in \mathcal{P}}\mathbb{E}_P[X], \ X\in L_{ip}(\Omega_T).$$
 Then the associated  Choquet capacity is defined by
$$c(A)=\sup_{P\in \mathcal{P}}P(A), \ A\in \mathcal{B}(\Omega_T).$$
A set $A\subset \Omega_T$ is called polar if $c(A)=0$,  and we say that a property   holds  quasi-surely (q.s.)
if it holds outside a polar set.

In this paper,  we consider the following $G$-SDE
\begin{equation}\label{kerner*}
 \d X_t=b(t,X_t)\d t+ \sum_{i,j=1}^dh_{ij}(t,X_t)\d\langle
 B^i,B^j\rangle_t+\langle\sigma(t,X_t),\d B_t\rangle,
\end{equation}
where $b, h_{ij}=h_{ji}:[0,T]\times   \R^d\to \R^d$ and $\si:[0,T]
\times\R^d\to \R^d\otimes \R^d$,  $B_t$ is a $d$-dimensional
$G$-Brwonian motion, and $\langle
 B^i,B^j\rangle_t$ stands for the mutual variation process of the
 $i$-th component $B^i_t$ and the
 $j$-th component $B^j_t$. To  ensure  the existence and uniqueness of the solution of
\eqref{kerner*} in $ M_G^2([0,T]; \R^d)$,
 we  assume
 \begin{equation}\label{Lip}
    \begin{split}
    & |b(t, x)-b(t, y)|+\sum_{i,j=1}^d |h_{ij}(t, x)-h_{ij}(t, y)|+\|\sigma(t, x)-\sigma(t, y)\|_{\rm HS}\leq
    K|x-y|,
     \end{split}
    \end{equation}
    for some   constant $  K\geq0$ and all $t\in[0,T]$, $ x$, $y\in \mathbb{R}^d;$
    see \cite[Theorem 1.2, p82]{peng4}.

Now we recall from \cite{RW} the following  notions for the path independence of additive functionals.

\begin{defn}\label{def}
For $ f=(f_{ij}):\R_+\times\R^d\rightarrow\mathbb{S}^d$ and
$g:\R_+\times\R^d\rightarrow\R^d$, the additive functional
$(A_{s,t}^{f,g})_{0\leq s\leq t}$  is defined by
\begin{equation}\label{eq2}
  A_{s,t}^{f,g}=\beta\int_s^t G(f)(r, X_r)\d r+\aa\sum_{i,j=1}^d\int_s^t  f_{ij}(r, X_r)\d\langle B^i,B^j\rangle_r +\int_s^t\langle g(r, X_r),
  \d B_r\rangle,
\end{equation}
where  $ \beta,\aa \in\mathbb{ R}$ are two parameters, $f=f^\ast$, and $(X_r)_{r\ge0}$ solves \eqref{kerner*}.
\end{defn}

\begin{defn}\label{defn1}
The additive functional $ A_{s,t}^{f,g}$ is said to be
path independent, if there exists a  function $
V:\R_+\times\R^d\rightarrow\R $ such that for any  $s\in [0,T]$ and any solution $(X_t)_{t\in[s,T]}$ to  \eqref{kerner*} from time $s$,
it holds
\begin{equation}\label{eq1}
    A_{s,t}^{f,g} =V(t,X_t)-V(s,X_s), \ t\in[s,T].
\end{equation}
\end{defn}
In terms of {\bf Definition} \ref{defn1}, the path independence  of the additive functional $A_{s,t}^{f,g}$ means that  $A_{s,t}^{f,g}$ depends only on $X_s$ and $X_t$  but not the path $(X_r)_{s<r< t},$ for any solution $(X_r)_{r\in[s,T]}$ to  \eqref{kerner*} from times $s$ and any $t\in(s, T]$.

The aim of this paper is to provide sufficient and necessary  characterizations for the path independence of the additive
functional $A_{s,t}^{f,g}$.

 To see that \eqref{eq2} covers additive functionals investigated in existing references for the path independence under
 the linear probability space, let $\Theta$ in \eqref{GA} be a singleton: $\Theta={Q}$, and $\bar{\mathbb{E}}=\E$ be a linear expectation.   Then the associated $G$-Brownian motion $B_t$ becomes the classical zero-mean normal distributed with covariance $Q$. Specially,
 let  $\underline{\sigma}^2=\overline{\sigma}^2={\bf1}_{d\times d}$,  i.e., $Q={\bf1}_{d\times d}$,  $G(A) =\frac{1}{2}\mbox{trace}(A)$, $A\in \mathbb{S}^d$,
we have  $\d \langle B^i,B^j\rangle_r=\delta_{ij}\d r$, where
$\delta_{ij}$ is a indicative function, $1\leq i,j \leq d$,
 and \eqref{eq2} reduces to
 \begin{equation*}
  A_{s,t}^{f,g}=\left(\aa+\frac{\beta}{2}\right)\sum_{i=1}^d\int_s^t  f_{ii}(r, X_r)\d r +\int_s^t\langle g(r, X_r),
  \d B_r\rangle.
\end{equation*}
 Taking ${\bf{f}}(r, X_r)=\left(\aa+\frac{\beta}{2}\right)\sum_{i=1}^d f_{ii}(r, X_r)$, this goes back to the additive functional studied in \cite{RW}:
\begin{equation*}
  A_{s,t}^{f,g}:=\int_s^t {\bf{f}}(r, X_r)\d r +\int_s^t\langle g(r, X_r),
  \d B_r\rangle.
\end{equation*}
In particular, when ${\bf{f}}=\frac{1}{2}|g|^2$, we have
\begin{equation}\label{eq3}
  A_{s,t}^{\frac{1}{2}|g|^2,g}
 =\frac{1}{2}\int_s^t|g|^2(r, X_r)\d r+\int_s^t\langle g(r, X_r),\d B_r\rangle , \ 0 \leq s\leq t,
\end{equation}
which corresponds to the Girsanov transform $\d\mathbb{Q}_{s,t}:=\exp\{{-A_{s,t}^{\frac{1}{2}|g|^2,g}}\}\d\mathbb{P}$.
To make the solution $X_t$ of \eqref{kerner*}    a martingale under $\mathbb{Q}_{s,t}$,   we reformulate \eqref{kerner*} as
\begin{equation*}
 \d X_t=\{b+ h\}(t,X_t)\d t+\langle\sigma(t,X_t),\d B_t\rangle,
\end{equation*}
where \begin{equation}\label{eq4}
 h(t,u):=\sum_{i=1}^dh_{ii}(t,u),
~~ (t,u)\in\R_+\times\R^d.
\end{equation}
When $\sigma$ is invertible,   taking $g=\si^{-1}(b+h)$  in \eqref{eq3},  we have
\begin{equation*}
  A_{s,t}^{\frac{1}{2}|g|^2,g}=\frac{1}{2}\int_s^t|\si^{-1}(b+h)|^2(r,X_r)\d r+\int_s^t\langle (\si^{-1}(b+h))(r,X_r),\d B_r\rangle , \ 0 \leq s\leq
 t.
\end{equation*}
Then, by the Girsanov theorem, $(X_r)_{s \leq r\leq t}$ is a martingale under $\mathbb{Q}_{s,t}$, which fits well the requirement of risk netural measure in finance.
The path independence of this particular additive functional has been investigated in
\cite{Wang3, RW,Wang,Wang1,Wang2}.

\begin{rem} \label{rem1}
    When $\alpha \neq 0$,  \eqref{eq2}  is  equivalent to
    \begin{equation*}\label{md}
    \begin{split}
    \alpha^{-1} A_{s,t}^{f,g}&=\alpha^{-1}\beta \int_s^t G(f)(r, X_r)\d r+\sum_{i,j=1}^m\int_s^t  f_{ij}(r, X_r)\d\langle B^i,B^j\rangle_r\\
    &\quad +\int_s^t\langle \alpha^{-1}g(r, X_r),\d B_r\rangle.
    \end{split}
    \end{equation*}
    So, in this case, the path independence of the  additive functional \eqref{eq2} can be reduced to the case of $\alpha=1.$
However, the case for $\aa=0$ also includes  interesting examples (see  Example \ref{exa1} below), so it is reasonable to consider $A_{s,t}^{f,g}$ in \eqref{eq2} with  two parameters $\alpha $ and $ \beta$.
\end{rem}

The remainder of the paper is organized as follows. In Section 2, following the line of \cite{Song2, Song}, we present a decomposition theorem for
multidimensional  $G$-semimartingales. In Section 3,
 we   characterize the path independence of $A^{f,g}_{s,t}$ using  nonlinear PDEs, so that  main results in \cite{Wang3,Wang,Wang1,Wang2} are extended to the present
  nonlinear  expectation setting.  Finally, in Section 4, we provide an example
to  illustrate the main result for $\alpha=0$ as mentioned in Remark
\ref{rem1}.

\section{A   Decomposition Theorem}
This part is essentially due to  \cite{Song2, Song}.  Set
$\delta_n(t):=\sum_{i=1}^{n-1}(-1)^i{\bf 1}_{(\frac{i}{n},
\frac{i+1}{n}]}(t)$, $t\in[0,T].$
 For any $A, B\in \mathbb{S}^d$, let $(A,B)_{\rm
 HS}=\mbox{trace}(AB)$ and $\|A\|_{\rm HS}=\ss{(A,A)_{\rm
 HS}}$. Then $(\mathbb{S}^d, \<\cdot,\cdot\>_{\rm HS}, \|\cdot\|_{\rm
 HS})$ is a Hilbert space; see e.g. \cite{RWu}.  Let   $\mbox{\mbox{spec}}(\cdot)$ be    the spectrum of a
 matrix $\cdot$, and let $\<B\>_t=(\<B^i,B^j\>_t)_{ij}$.

From now on, we consider
\begin{equation}\label{G(A)}
 G(A):=\frac{1}{2}\sup _{\gamma\in [\underline{\sigma}, \bar{\sigma}]}\<\gamma^2,A\>_{\rm HS}, \ A\in\mathbb{S}^d, 0<\underline{\sigma}<\overline{\sigma} \ \rm{are \ two\ matrices\   in} \ \mathbb{S}^d.
\end{equation}
Consequently, $\underline{\sigma}^2<\frac{\d}{\d t}\langle B\rangle_t\leq\bar{\sigma}^2,$ and \eqref{Gnon} holds for $\delta=\lambda_0(\underline{\sigma}^2)$, where $\lambda_0(\underline{\sigma}^2)=\min\{\ll\in\mbox{spec}(\underline\si^2)\}$.

Let $  M_G^1([0,T]; \mathbb{S}^d) $  be all symmetric $d\times d$ matrices $\eta_t=(\eta_t^{ij})_{d\times d}$ with $\eta^{ij}\in M_G^1([0,T]), $  and
$\|\eta\|_{M_G^1([0,T]; \mathbb{S}^d)}=\bar{\mathbb{E}}\int_{0}^{T}\|\eta_{t}\|_{\rm HS}\d t.$

  Let
 $c_0=\min\{\ll\in\mbox{spec}((\bar\si^2-\underline\si^2)/2)\}$,  $C_0=\ff{1}{2}\|\bar\si^2-\underline\si^2\|_{\rm
 HS}$.

 To make the content
 self-contained, we cite from \cite{Song2} some  well-known  results and
 restated them as follows.

\begin{lem}\label{lem01}
{\rm Let $G$ be in \eqref{G(A)}.  For any $\eta\in M_G^1([0,T]; \mathbb{S}^d)$, the
limit
$$\|\eta\|_{\mathbb{M}_G}:=\lim_{n\to \infty}\bar{\mathbb{E}}  \int_0^T \delta_n(s)(\eta_s,\d\langle
B\rangle_s)_{\rm HS}$$ exists. Moreover, $\|\cdot\|_{\mathbb{M}_G}$
defines a norm on $M_G^1([0,T]; \mathbb{S}^d), $  and for any $0<\epsilon\leq c_0$, it holds that,
$$\epsilon\|\eta\|_{M_{G_\epsilon}^1([0,T]; \mathbb{S}^d)}\leq \|\eta\|_{\mathbb{M}_G}\leq C_0\|\eta\|_{M_G^1([0,T]; \mathbb{S}^d)},$$
where
$G_\epsilon(A):=\frac{1}{2}\sup_{\gamma\in[\underline{\sigma}_{\epsilon},
\bar{\sigma}_{\epsilon}]}\<\gamma^2,A\>_{\rm HS},$
$\underline{\sigma}_{\epsilon}^2:=\underline{\sigma}^2+\epsilon\,{\bf
I}_{d\times d}$, and
$\bar{\sigma}_{\epsilon}^2:=\bar{\sigma}^2-\epsilon\,{\bf
I}_{d\times d}$. }
\end{lem}

With Lemma \ref{lem01} in hand, we have the following
corollary which will play a crucial role in the analysis below.




\begin{cor}\label{cor01} {\rm Let $G$ be in \eqref{G(A)}, and  let $\eta\in M_G^1([0,T]; \mathbb{S}^d),$ $\zeta \in M_G^1([0,T])$. If
 $$\int_0^t (\eta_s,\d\langle B\rangle_s)_{\rm HS}=\int_0^t \zeta_s\d s,~~~t\in[0,T],$$
 then
$$\bar{\mathbb{E}}  \int_0^T \|\eta_s\|_{\rm HS}\d s=\bar{\mathbb{E}} \int_0^T |\zeta_s|\d
s=0.$$}
\end{cor}

\begin{proof} According to  Lemma \ref{lem01} and   \cite[Theorem 3.3 (i)]{Song}, we
deduce that
\begin{equation*}
\|\eta\|_{\mathbb{M}_G}=\lim_{n\to \infty}\bar{\mathbb{E}}  \int_0^T
\delta_n(s)(\eta_s,\d\langle B\rangle_s)_{\rm HS} =\lim_{n\to
\infty}\bar{\mathbb{E}} \int_0^T \delta_n(s)\zeta_s\d s=0.
\end{equation*}
Recall from  Lemma \ref{lem01} that  $\|\eta\|_{\mathbb{M}_G}$ is a
norm, then, $c$-q.s.,  $\eta_t\equiv{\bf0}_{d\times d}$, a.e. $t\in
[0,T]$. Therefore, $\bar{\mathbb{E}}  \int_0^T \|\eta_s\|_{\rm HS}\d
s =0,$ which leads to $\bar{\mathbb{E}}  \int_0^T |\zeta_s|\d s =0$.
\end{proof}

 Consider the following It\^{o} process in
 $(\Omega_T,L_G^p(\Omega_T),\bar{\mathbb{E}})$
    \begin{equation}\label{Peng}
    X_t=\int_0^t \xi_r\d r+\sum_{i,j=1}^d\int_0^t \eta_r^{ij} \d\langle B^i,B^j\rangle_r +\sum_{i=1}^d\int_0^t \zeta_r^i \d
    B_r^i,~~~t\geq0,
    \end{equation}
where $\xi,  \eta^{ij} \in M_G^1([0,T];\R^d)$ with $\eta^{ij}
=\eta^{ji}$,   and $\zeta^i\in
    H_G^1([0,T];\R^d)$.

Now we can state the following decomposition theorem.
\begin{thm}\label{thm} \label{Thm2.3}
 For $G$   in \eqref{G(A)} and let $X_t $ be in \eqref{Peng}. Then
   $X_t={\bf0}_d$ for all $t\in[0,T] $ if and only if on $\Omega_T\times [0,T]$ it holds  $c\times \d t-q.s. \times a.e.$,
    $\xi_t={\bf0}_d,  \eta_t^{ij}={\bf0}_{d}, \zeta_t^i={\bf0}_{d}$, $i,j=1,
    \cdot\cdot\cdot,d.$
\end{thm}

\begin{proof}
The proof of the sufficiency is trivial, it suffices to prove
the necessity. Assume  $X_t={\bf0}_{d}$ for   $t\in[0,T]$. Then \eqref{Peng} is
equivalent to
\begin{equation}\label{reduced}
     \int_0^t\xi_s^k\d s+ \sum_{i,j=1}^d\int_0^t \eta_r^{kij} \d\langle B^i,B^j\rangle_r +\sum_{j=1}^d\int_0^t \zeta_s^{kj} \d B_s^j={\bf0}_{d},
     ~k=1,\cdots,d,   \ t\in[0,T],
    \end{equation}
where $\xi^k_\cdot$ (resp. $\eta_\cdot^{kij})$ denotes the $k$-th
component of the column vector $\xi$ (resp. $\eta_\cdot^{ij})$.
Taking quadratic processes w.r.t. $\int_0^\cdot \zeta_s^{ki} \d
B_s^i $ on both side of \eqref{reduced}, we deduce
that
\begin{equation*}
\begin{split}
    {0}&= \sum_{i,j=1}^d  \left\langle \int_0^\cdot \zeta_s^{kj} \d B_s^j,\int_0^\cdot \zeta_s^{ki} \d B_s^i \right\rangle_t = \left\langle \sum_{i=1}^d
    \int_0^\cdot \zeta_s^{ki }\d B_s^i \right\rangle_t\\
&= \sum_{i,j=1}^d\int_0^t \zeta_s^{kj} \zeta_s^{ki} \d\langle B^i, B^j\rangle_s=\int_0^t ((\zeta_s^k)^*\zeta_s^k,  \d\langle B \rangle_s)_{\rm HS}\\
&=\int_0^t \langle  \d\langle B \rangle_s(\zeta_s^k )^*,
(\zeta_s^k)^*\rangle \geq\int_0^t \langle  \underline{\sigma}^2 (\zeta_s^k )^*,
(\zeta_s^k)^*\rangle \d s
\geq0
\end{split}
    \end{equation*}
with $\zeta^k:=(\zeta^{k1},\cdots,\zeta^{kd})$.
Since $\underline{\sigma}^2>0,$ this
   implies   $\zeta_t={\bf0}_{d\times d}$, a.e.
$t\in [0,T]$.

It remains to show that $\xi_t={\bf0}_{d}$ and $\eta_t={\bf0}_{d}$.
In fact, since $\zeta_t={\bf0}_{d}$, we have
 \begin{equation*}
 \int_0^t-\xi_s^k\d s= \sum_{i,j=1}^m\int_0^t \eta_s^{kij} \d\langle B^i,B^j\rangle_s=  \int_0^t(\eta_s^k, \d\langle B\rangle_s)_{\rm HS},
  \ k=1, \cdot\cdot\cdot, d,~~t\in[0,T]
    \end{equation*}
with $\eta_\cdot^k=(\eta_\cdot^{kij})_{ij}$. By   Corollary
\ref{cor01}, this implies
 \begin{equation*}
 \bar{\mathbb{E}}\int_0^T|\xi_s^k|\d s=  \bar{\mathbb{E}} \int_0^T\|\eta_s^k\|_{\rm HS} \d s=0, \ k=1, \cdot\cdot\cdot, d.
    \end{equation*}
Thus, we conclude that $c$-q.s. for a.e.  $t\in [0,T],$ $\eta_t^k=
{\bf0}_{d\times d}$ and $\xi_t^k=0, \ k=1, \cdot\cdot\cdot, d.$
Therefore, $\xi_t=\eta_t^{ij}={\bf0}_{d}$.

\end{proof}

\section{Characterization of Path Independence}



The main result of the paper is the  following.
\begin{thm}\label{theorem1}
Let $G$ be in \eqref{G(A)}.  Then
 $A_{s,t}^{f,g}$ is path independent in the sense of \eqref{eq1} for some $V\in
C^{1,2}(\mathbb{R}_{+}\times\mathbb{R}^d)$
 if and only if
 \begin{equation}\label{thm2}
 \begin{cases}
 \frac{\partial }{\partial t}V(t,x)
 =\beta G(f)(t,x)-\<\nn V,b\>(t,x),\\
\aa f_{ij}(t,x)=\Big(\<\nn V,h_{ij}\>+ \frac{1}{2}\<\si_i,(\nn^2V)\si_j\>  \Big)(t,x),   \\
 g(t,x)=(\si^*\nn V)(t,x),   \ \  \ (t,x) \in [0,T]\times \mathbb{R}^d,
\end{cases}
 \end{equation}
 where $i,j=1,\cdots,d$, and $\sigma_i$ stands for the $i$-th column of   $\sigma$.
 \end{thm}
 \begin{proof} We first prove the necessity.  For any $(s,x)\in[0,T]\times \R^d$, let  $(X_t)_{t\geq s} $ solves \eqref{kerner*} with $X_s=x$.
 Since $(A_{s,t}^{f,g})_{t\in[s,T]}$ is path independent in the sense of \eqref{eq1}, it follows that
\begin{eqnarray}\label{2}
 \d V(t,X_t)&=&\beta G(f)(t,X_t)\d t+\aa\sum_{i,j=1}^d f_{ij}(t, X_t)\d\langle B^i,B^j\rangle_t\\ \nonumber
 &&+\langle g(t, X_t),\d B_t\rangle, \ t\in [s,T].
 \end{eqnarray}
On the other hand, by It\^{o}'s formula, we derive that
\begin{equation}\label{1}
\begin{split}
 \d V(t,X_t)&=\left(\frac{\partial }{\partial t}V+\<\nn V,b\> \right)(t,X_t)\d t +\<(\si^*\nn V)(t,X_t),\d B_t\> \\
 &\quad+\sum_{i,j=1}^d\Big(\<\nn V,h_{ij}\>+ \frac{1}{2}\<\si_i,(\nn^2V)\si_j\>  \Big)(t,X_t)\d \langle B^i, B^j\rangle_t, \ t\in [s,T].
\end{split}
\end{equation}
Since coeffieients $b,h$ and $\sigma$ satisfy the Lipschitz condition in \eqref{Lip}, and the solution of \eqref{kerner*} satisfies  $X_t\in  M_G^2([0,T]; \R^d)$,  it's  not difficult to verify
$\left(\frac{\partial }{\partial t}V+\<\nn V,b\> \right)(t,X_t)\in M_G^1([0,T]) $, $\Big(\<\nn V,h_{ij}\>+ \frac{1}{2}\<\si_i,(\nn^2V)\si_j\>  \Big)(t,X_t)\in  M_G^1([0,T])$,
and $(\si_i^\ast\nn V)(t,X_t)\in H_G^1([0,T])$,
thus   hypotheses of Theorem \ref{Thm2.3} are satisfied.
Combining \eqref{2} and \eqref{1}, and applying  Theorem  \ref{thm}
for the process $V(t,X_t)$, we obtain  $c$-q.s. for  a.e. $t\in[s,T]$,
 \begin{equation}\label{3}
 \begin{cases}
 \Big(\frac{\partial }{\partial t}V+\<\nn V,b\> \Big)(t,X_t)=
  \beta G(f)(t,X_t),\\
\aa f_{ij}(t,X_t)=\Big(\<\nn V,h_{ij}\>+ \frac{1}{2}\<\si_i,(\nn^2V)\si_j\>  \Big)(t,X_t),   \\
 g(t,X_t)=(\si^*\nn V)(t,X_t).
\end{cases}
\end{equation}
Since all terms in \eqref{3} are continuous in $t$, these equations hold $c$-q.s. at $t=s$, so by $X_s=x$, we have
 \begin{equation}\label{4}
 \begin{cases}
 \Big(\frac{\partial }{\partial t}V+\<\nn V,b\> \Big)(s,x)=
  \beta G(f)(s,x),\\
\aa f_{ij}(s,x)=\Big(\<\nn V,h_{ij}\>+ \frac{1}{2}\<\si_i,(\nn^2V)\si_j\>  \Big)(s,x),   \\
 g(s,x)=(\si^*\nn V)(s,x).
\end{cases}
\end{equation}
Due to the arbitrariness of $s$ and $x$, we prove \eqref{thm2}.

Next, for  the sufficiency, taking advantage of \eqref{thm2}, we deduce from \eqref{1} that \eqref{2} holds true. By taking stochastic integration  we prove \eqref{eq1}, and   therefore complete the proof.
\end{proof}

Let us comparison this result with known ones in the linear expectation setting.
\begin{rem}\label{rem}
Comparing with  the Girsanov transform in the linear expectation
setting as mentioned in Introduction, we take for instance
$\alpha=1, \beta=-1$, $h_{ij}=0$ and
\begin{equation}\label{ast}
 f_{ii}=\frac{1}{d}|g|^2, \ 1\leq i\leq d.
\end{equation}
 When $\underline{\sigma}^2=\bar{\sigma}^2={\bf1}_{d\times d}$,  this goes back to  the classic linear expectation, $(B_t)_{t\ge0}$ is a $d$-dimensional standard Brownian motion defined
on the probability space $(\Omega,\F,\P)$,  we have $\langle B^i,B^j \rangle_r=\delta_{ij}r$,  and
\begin{equation*}\label{G}
  G(f)=\frac{|g|^2}{2d}\mbox{trace}[{\bf{1}}_{d\times d}]=\frac{{|g|^2}}{2}.
\end{equation*}
So
\begin{eqnarray*}\label{A}
A_{s,t}^{{ f},g}:&=&\beta \int_s^t G(f)(r, X_r)\d r+\alpha\sum_{i,j=1}^d\int_s^t  f_{ij}(r, X_r)\d\langle B^i,B^j\rangle_r\\ \nonumber
&&+\int_s^t\langle g(r, X_r),\d B_r\rangle\\ \nonumber
&=& \frac{1}{2}\int_s^t |g(r, X_r)|^2\d r
+\int_s^t\langle g(r, X_r),\d B_r\rangle
\end{eqnarray*}
gives the weighted probability $\exp\{{-A_{s,t}^{f,g}}\}\d\mathbb{P}$ in the Girsanov theorem.

 By taking   $\alpha=1, \beta=-1$, $h_{ij}=0$ and $f$ in  \eqref{ast},
 the assertion of Theorem \ref{theorem1} becomes that  $A_{s,t}^{f,g}$ is path independent in the sense of \eqref{eq1} for some  $V\in C^{1,2}(\R_+\times\R^d;\R)$   if and only if
\begin{equation*}\label{8}
 \begin{cases}
 \frac{\partial }{\partial t}V(t,x)
 =-\frac{1}{2} G\Big((\<\si_i,(\nn^2V)\si_j\> )_{1\leq i,j\leq d}\Big)(t,x)-\<\nn V,b\>(t,x),\\
 f_{ij}(t,x)=\frac{1}{2}\<\si_i,(\nn^2V)\si_j\>  (t,x),   \\
 g(t,x)=(\si^*\nn V)(t,x),   \ \  \ (t,x) \in [0,T]\times \mathbb{R}^d.
\end{cases}
 \end{equation*}
 It is easy to see that this generalizes the main results derived  in
\cite{Wang3, Wang,Wang1,Wang2} where $h\equiv0$ and $g$ is given by $\si^{-1}b$, under additional condition ensuring the existence of $\si^{-1}b$, i.e., $b$ takes value in $\{  \si v: v\in \R^d\}$.

However, since $G$ is a linear function in the linear expectation case,  Theorem \ref{theorem1} does not directly apply to existing results,  but extends  them to the non-degenerate $G$-setting.
\end{rem}


Moreover, the nonlinear PDE included in \eqref{theorem1} covers the $G$-heat equation as a special example.

\begin{rem}
When  $h=b=0$, $\alpha=1$, and $\beta=-2$, the PDE  in \eqref{thm2} for $V$ reduces to the following $G$-heat equation
$$\frac{\partial }{\partial t} {V}(t,x)
+G\Big(\big(\<\si_i,(\nn^2V)\si_j\> (t,x)\big)_{1\leq i,j\leq m}\Big)=0,$$
which is one of main motivations for the study of $G$-Brownian motion.
\end{rem}
\section{An   Example with $\alpha=0$}
Now we provide an example to demonstrate our main result for $\alpha=0$. As indicated in Remark \ref{rem1} that when $\alpha\neq0 $ the study can be reduced to  $\alpha=1.$

\begin{exa} \label{exa1}
 Let $d=1$, $\alpha=0,$  and $ \beta=2.$    By Theorem \ref{theorem1}, $A_{s,t}^{f,g}$ is path independent if and only if
\begin{equation} \label{a}
 \begin{cases}
 f(t,x)= \frac{1}{2}G^{-1}\left(\frac{\partial V}{\partial t}+b\frac{\partial V}{\partial x}\right)(t,x),\\
\left(h\frac{\partial V}{\partial x}\right)(t,x)+\frac{1}{2}\left(\si^2\frac{\partial^2 V}{\partial x^2}\right)(t,x)=0,\\
 g(t,x)=\left(\si^*\frac{\partial V}{\partial x}\right)(t,x),    \ \  \ (t,x) \in [0,T]\times \mathbb{R}.
\end{cases}
 \end{equation}
We may solve $V$ by using $\mathcal{L}_t$-Harmonic function:
\begin{equation}\label{ex1}
  \mathcal{L}_tV_0(x)=0, \ V_0\in C^{1,2}(\R\to\R), \ t\geq0,
\end{equation}
where  $\mathcal{L}_t=h(t,x) \frac{\partial }{\partial x}+\frac{1}{2}\si^2(t,x)\frac{\partial^2 }{\partial x^2}$.

For any $\mathcal{L}_t$-Harmonic function $V_0$,  $t\geq0,$ let $V(t,x)=\varphi(t)V_0(x)$ for some $\varphi\in C^{1,2}(\R_+\to\R)$. Then $V$ solves the above PDE  in \eqref{a}.
Therefore, $A_{s,t}^{f,g}$ is path independent if
\begin{equation}\label{ex2}
 \begin{cases}
 f(t,x)= \frac{1}{2}G^{-1}\left(\varphi'(t)V_0(x)
+b(t,x)\varphi(t) V_0'(x)\right),\\
 g(t,x)=\si(t,x)\varphi(t) V_0'(x),   \ \  \ (t,x) \in [0,T]\times \mathbb{R}.
\end{cases}
 \end{equation}
To present specific choices of $V_0$, let $h$ and $\si$ do not depend on $t$. Then \eqref{ex1} becomes
\begin{equation*}
  h(x)V_0'(x)+\frac{1}{2}\si^2(x)V_0''(x)=0.
\end{equation*}
 When $\si^2(x)\neq 0,$ this is equivalent to
 \begin{equation*}
 V_0''(x)=-2\frac{h(x)}{\si^2(x)}V_0'(x).
 \end{equation*}
Thus,
\begin{equation*}
  V_0(x)=V_0(0)+V_0'(0)\int_0^x\e^{-2\int_0^u\frac{h(r)}{\si^2(r)}\d r}\d u.
\end{equation*}
In particular, when $\si(x)=1, h(x)=x$, we have
\begin{equation*}
  V_0(x)=V_0(0)+V_0'(0)\int_0^x\e^{-u^2}\d u,
\end{equation*}
which is related to the Gaussian distribution.
\end{exa}

\paragraph{Acknowledgement.} The authors are grateful  to  Professor Feng-Yu Wang for his guidance, valuable suggestion and comments on  earlier versions of the paper,
as well as Professor Yongsheng Song for his patient help and corrections.

\beg{thebibliography}{99}

\bibitem{B08} Bishwal J P N, Parameter Estimation in Stochastic Differential Equations,  Springer, Berlin, 2008.

\bibitem{BS} Black F, Scholes M, The pricing of options and corporate liabilities,
 J   Polit  Econ, 1973, 81: 637--654.

\bibitem{BR} Bai X P, Buckdahn R, Inf-convolution of $G$-expectations,  Sci China Math,  2010, 53: 1957-1970.

\bibitem{15} Denis L, Hu M S, Peng S G, Function spaces and capacity related to a sublinear expectation: application to $G$-Brownian motion pathes, Potential Anal, 2011, 34: 139--161.

\bibitem{FWZ} Feng C R,  Wu P Y,  Zhao H Z,
Ergodicity of Invariant Capacity, 2018, arXiv: 1806.03990.

\bibitem{FOZ} Feng C R,  Qu B Y,  Zhao H Z,  A Sufficient and Necessary Condition of PS-ergodicity of Periodic Measures and Generated Ergodic Upper
Expectations, 2018,  arXiv: 1806.03680.

\bibitem{GRR} Gu Y F, Ren Y,  Sakthive R, Square-mean pseudo almost automorphic mild solutions for stochastic evolution equations driven by $G$-Brownian motion,  Stoch Anal  Appl, 2016, 34: 528--545.

\bibitem{HJP}
 Hu M S,  Ji S L, Peng S G, Song Y S, Comparison theorem, Feynman-Kac formula and Girsanov transformation for BSDEs driven by $G$-Brownian motion,
 Stochastic Process Appl, 2014, 124: 1170--1195.

 \bibitem{HJ}   Hu M S,  Ji S L, Dynamic programming principle for stochastic recursive optimal control problem driven by a $G$-Brownian motion,
Stochastic Process Appl, 2017, 127:  107--134.

\bibitem{HJY}  Hu M S,  Ji S L, Yang S Z, A stochastic recursive
optimal control problem under the $G$-expectation framework,
  Appl Math Optim, 2014, 70:  253--278.


\bibitem{peng2}  Peng S G, $G$-Brownian motion and dynamic risk measures under volatility
uncertainty, 2007, arXiv: 0711.2834v1.

\bibitem{peng1} Peng S G, $G$-expectation, $G$-Brownian motion and related stochastic
calculus of It\^{o} type, in: Stochastic Analysis and Applications, in: Abel Symp., vol. 2, Springer, Berlin, 2007, pp.541--567.

\bibitem{peng4} Peng S G,  Nonlinear expectations and stochastic calculus under uncertainty, 2010, arXiv: 1002.4546v1.

\bibitem{P} Peng S G, Theory, methods and meaning of nonlinear expectation theory (in Chinese), Sci Sin Math, 2017,  47: 1223--1254, doi: 10.1360/N012016-00209.

\bibitem{Song2} Peng S G, Song Y S, Zhang J F, A complete representation theorem for
$G$-martingales,  Stochastics,  2014,  86: 609--631.

\bibitem{Wang3} Wu J L, Yang W, On stochastic differential equations and a generalised Burgers
equation, pp 425--435 in Stochastic Analysis and Its Applications to Finance: Essays in
Honor of Prof. Jia-An Yan (eds. T S Zhang, X Y Zhou), Interdisciplinary Mathematical
Sciences, Vol. 13, World Scientific, Singapore, 2012.

\bibitem{QW} Qiao H J, Wu J L, Characterizing the path-independence of the Girsanov transformation for non-Lipschitz SDEs with jumps,  Statist Probab Lett, 2016, 119: 326--333.

\bibitem{QW2} Qiao H J, Wu J L,  On the path-independence of the Girsanov transformation for stochastic evolution equations with jumps in Hilbert spaces, 2018,
arXiv: 1707.07828.

\bibitem{RW} Ren P P, Wang F Y, Space-Distribution PDEs for Path Independent Additive Functionals
of McKean-Vlasov SDEs, 2018, arXiv: 1805.10841.

\bibitem{RWu} Ren P P, Wu J L,
Matrix-valued SDEs arising from currency exchange markets,
arXiv:1709.00471.

\bibitem{RYS}
 Ren Y, Yin W S, Sakthivel R,  Stabilization of
stochastic differential equations driven by $G$-Brownian motion with
feedback control based on discrete-time state observation,
Automatica J IFAC, 2018, 95: 146--151.

\bibitem{Song} Song Y S,  Uniqueness of the representation for $G$-martingales with finite variation, Electron J Probab, 2012, 17:  1--15.

\bibitem{Wang} Truman A,  Wang F Y, Wu J L, Yang W, A link of stochastic differential equations to nonlinear parabolic equations, Sci. China Math, 2012, 55: 1971--1976.

\bibitem{Wang1} Wang M, Wu J L, Necessary and sufficient conditions for path-independence of Girsanov transformation for infinite-dimensional stochastic evolution equations, Front
Math China, 2014, 9: 601--622.

\bibitem{WG} Wang B J, Gao H J, Exponential Stability of Solutions to Stochastic Differential Equations Driven by $G$-L\'{e}vy Process, 2018, arXiv: 1801.03776.

\bibitem{Wang2} Wu B, Wu J L, Characterizing the path-independent property of the Girsanov density
for degenerated stochastic differential equations, Stat Probab Letters, 2018, 133: 71--79.


\bibitem{HC} Hodges S, Carverhill A, Quasi mean reversion in an efficient stock market: the characterisation of economic equilibria which support Black-Scholes option pricing,  The Economic Journal,  1993, 103: 395--405.

\end{thebibliography}

\end{document}